\documentclass[11pt,a4paper]{article}
\usepackage{amsfonts, amsthm, amsmath, amssymb}
\usepackage[english]{babel}

\theoremstyle{plain}

\newtheorem*{theorem*}{Theorem}

\newtheorem{lemma}{Lemma}

\theoremstyle{definition}

\theoremstyle{remark}
\newtheorem{remark}{Remark}

\title{Weak approximation of an invariant measure and a low boundary
of the entropy}

\author{B.M. Gurevich\thanks{The work is supported in part
by the RFBR grant 13-01-12410.}\\ MSU and IITP RAS}

\date{}

\begin{document}
\maketitle

\bigskip

\begin{abstract}
For a measurable map $T$ and a sequence of $T$-invariant probability
measures $\mu_n$ that converges in some sense to a $T$-invariant
probability measure $\mu$, an estimate from below for the
Kolmogorov--Sinai entropy of $T$ with respect to $\mu$ is suggested
in terms of the entropies of $T$ with respect to $\mu_1$, $\mu_2$,
\dots.
\end{abstract}
\bigskip

In problems of Ergodic theory and Thermodynamic formalism it is
sometimes necessary to estimate the entropy of a measure preserving
map. If this map acts in a compact metric space and is expansive,
one can use the fact that the entropy is semicontinuous from above
on the space of invariant probability measures. Both conditions ---
the compactedness and expansiveness are essential in this context,
and if at least one of them fails, one has to use other means. One
such a mean is suggested in this note.

We will use standard notation, terminology and results from Entropy
theory (see, e.g., \cite{CFS} -- \cite{R}). Let $T$ be an
automorphism of a measurable space $(X,\mathcal F)$ and $\mu_0$ a
$T$-invariant probability measure. For a finite or countable
infinite partition $\eta$ of $(X,\mathcal F)$, we write $B\in\eta$
and $B\subset\eta$ if the set $B$ is an atom of $\eta$ or a union of
such atoms, respectively.
\begin{theorem*}
\label{main} Assume that for a countable measurable partition $\xi$
of $(X,\mathcal F)$, the entropy $H_{\mu_0}(\xi)$ is finite and that
there exist a sequence of $T$-invariant probability measures $\mu_n$
and sequences of numbers $r_n\in\mathbb N$, $\varepsilon_n>0$ such
that
\begin{equation}
\label{limits} \lim_{n\to\infty}r_n=\infty,\ \
\lim_{n\to\infty}\varepsilon_n=0,\
\limsup_{n\to\infty}h_{\mu_n}(T)\ge h\ge 0,
\end{equation}
\begin{equation}
\label{generator} \xi\text{\ is a generator for\ } (T,\mu_n),\
n\ge0,
\end{equation}
\begin{equation}
\label{uslovie} |\mu_0(A)-\mu_n(A)|\le\varepsilon_n\mu_n(A) \text{\
for all\ } A\in\vee_{i=0}^{r_n}T^{-i}\xi,\ n\ge0.
\end{equation}
Then $h_{\mu_0}(T)\ge h$.
\end{theorem*}
We begin the proof of the Theorem with two simple lemmas.
\begin{lemma}
\label{entrop_razn} Let $\mathbf p:=(p_i)_{i\in\mathbb N}$, $\mathbf
q:=(q_i)_{i\in\mathbb N}$, where $p_i,q_i\ge 0$ for all $i$ and
$\sum_ip_i=\sum_iq_i=1$. Let also $H(\mathbf p):=-\sum_{i\in\mathbb
N}p_i\ln p_i$ (with $0\ln 0=0$) and $H(\mathbf q)$ be defined
similarly. Assume that for some $c\in(0,1/3)$,
\begin{equation}
\label{sravn_ver} |p_i-q_i|\le cq_i,\ \ i=1,2,\dots.
\end{equation}
Then
\begin{equation*}
\label{entrop_razn1} H(\mathbf p)\le(1+c)H(\mathbf q)+c\ln 3.
\end{equation*}
\end{lemma}
\begin{proof}
Denote $\varphi(t):=-t\ln t$, $t\ge 0$. It is clear that (a)
$\varphi(t)$ increases when $0\le t\le e^{-1}$, (b) $\varphi(t)\le
0$ when $t\ge 1$, (c) $-1\le\varphi'(t)\le\ln 3$ when $(3e)^{-1}\le
t\le 1$. Hence (see also \eqref{sravn_ver})
\begin{align*}
H(\mathbf p)=&\sum_{i\in\mathbb N}\varphi(p_i)=\sum_{i:q_i\le
1/2e}\varphi(p_i)+\sum_{i:q_i>1/2e}\varphi(p_i)\le\sum_{i:q_i\le
1/2e}\varphi((1+c)q_i)\notag \\
+&\sum_{i:q_i>1/2e}[\varphi(q_i)+|p_i-q_i|\ln 3]=\sum_{i:q_i\le
1/2e}[q_i\varphi(1+c)+(1+c)\varphi(q_i)] \notag \\
+&\sum_{i:q_i>1/2e}[\varphi(q_i)+|p_i-q_i|\ln
3]\le(1+c)\sum_{i\in\mathbb N}\varphi(q_i)+c\ln 3=(1+c)H(\mathbf q)+
c\ln3.
\end{align*}
\end{proof}
\begin{lemma}
\label{AinEta} If $\eta$ is a countable measurable partition of the
space $(X,\mathcal F)$ and if, for probability measures $\mu$ and
$\nu$ on $(X,\mathcal F)$, for every $A\in\eta$ and some
$\varepsilon>0$, we have $|\mu(A)-\nu(A)|\le\varepsilon\nu(A)$, then
the same is true for every $A\subset\eta$.
\end{lemma}

The proof is evident and will be omitted.
\medskip

We continue the proof of the Theorem. For all $k,l\in\mathbb Z$,
$k\le l$, we denote $(\xi^T)_{-l}^{-k}:=\vee_{i=l}^kT^{-i}\xi$.

It is known \cite{R} that if $H_{\mu_0}(\xi)<\infty$ and
\eqref{generator} holds for $n=0$, then
\begin{equation*}
\label{lim_entrop}
h_{\mu_0}(T)=\lim_{n\to\infty}\frac{1}{n}H_{\mu_0}((\xi^T)_{-n}^{-1}),
\end{equation*}
and the sequence on the right hand side is non-increasing.

It easy to verify that if $\varepsilon\le 1/2$, then
$|a-b|\le\varepsilon b$, $a,b\ge 0$ imply that $|a-b|\le
2\varepsilon a$. Therefore by \eqref{uslovie} for all $n$ such that
$\varepsilon_n\le1/2$ and all $A\in\xi(-r_n,0)$, we have
\begin{equation}
\label{uslovie1} |\mu_0(A)-\mu_n(A)|\le2\varepsilon_n\mu_0(A).
\end{equation}

For these $n$, we compare $H_{\mu_0}((\xi^T)_{-r_n}^{-1})$ and
$H_{\mu_n}((\xi^T)_{-r_n}^{-1})$.

An arbitrary numbering of the atoms $A\in(\xi^T)_{-r_n}^{-1}$ yields
a sequence $A_1,A_2,\dots$. Let $p_i:=\mu_n(A_i)$,
$q_i:=\mu_0(A_i)$. By applying Lemma \ref{entrop_razn} for
$c=2\varepsilon_n$ (see \eqref{uslovie1}) we obtain
\begin{equation*}
\label{entrop_razn2}
H_{\mu_n}((\xi^T)_{-r_n}^{-1})\le(1+2\varepsilon_n)
H_{\mu_0}((\xi^T)_{-r_n}^{-1})+2\varepsilon_n\ln 3.
\end{equation*}

For all sufficiently large $n$, this implies that
\begin{align*}
\label{osn_ner}
h_{\mu_n}(T,\xi)\le&\frac{1}{r_n}H_{\mu_n}((\xi^T)_{-r_n}^{-1}) \notag\\
\le&\frac{1}{r_n}(1+2\varepsilon_n)H_{\mu_0}((\xi^T)_{-r_n}^{-1})+
\frac{2}{r_n}\varepsilon_n\ln 3,
\end{align*}
or
\begin{equation*}
\label{osn_ner1} h_{\mu_n}(T,\xi)\le
\frac{1}{r_n}(1+2\varepsilon_n)H_{\mu_0}((\xi^T)_{-r_n}^{-1})+
\frac{2}{r_n}\varepsilon_n\ln 3.
\end{equation*}
Therefore (see \eqref{limits})
$$
h\le\limsup_{n\to\infty}h_{\mu_n}(T,\xi)\le h_{\mu_0}(T).
$$
The proof is completed.
\begin{remark}
\label{zamech} Let us number in an arbitrary way the atoms of the
partition $\xi$, and consider the finite partition $\xi_m$ obtained
from $\xi$ by replacing all atoms of  $\xi$ with labels $\ge m$ by
their union. In the statement of the Theorem, the conditions
$\limsup_{n\to\infty}h_{\mu_n}(T)\ge h$ (see \eqref{limits}) and
$H_{\mu_0}(\xi)<\infty$ can be changed for the single (and weaker)
condition
\begin{equation*}
\label{weak_cond}
\limsup_{m\to\infty}\limsup_{n\to\infty}h_{\mu_n}(T,\xi_m)\ge h.
\end{equation*}

\end{remark}

\end{document}